\documentclass[reqno]{amsart}
\usepackage{amssymb,amsmath,latexsym,amsthm,graphicx}
\usepackage{verbatim}

\newtheorem{thm}{Theorem}[section]
\newtheorem{prop}[thm]{Proposition}

\theoremstyle{definition}

\theoremstyle{remark}

\begin{document}

\title
{The Super Catalan Numbers $S(m, m + s)$ for $s \leq 4$}

\author{Xin Chen}
\address{Carleton College, Northfield, MN 55057 USA}
\email{chenx@carleton.edu}

\author{Jane Wang}
\address{Princeton University, Princeton, NJ 08544 USA}
\email{jywang@princeton.edu}

 \begin{abstract}
 We give a combinatorial interpretation using lattice paths for the super Catalan number $S(m, m+s)$ for $s \leq 3$ and a separate interpretation for $s = 4$.
\end{abstract}

\maketitle

\section{Introduction}
The Catalan numbers
\begin{equation}
\label{eq:catalan}
C_n =  \frac{1}{n+1} \binom{2n}{n} = \frac{(2n)!}{n! (n+1)!},
\end{equation}
are known to be integers and have many combinatorial interpretations.

In 1874, E. Catalan \cite{C} observed that the numbers
\begin{equation}
\label{eq:supercat}
S(m, n) = \frac{(2m)! (2n)!}{m! n! (m+n)!},
\end{equation} are also integers. Gessel \cite{G} later referred to these numbers as the super Catalan numbers since $S(1,n)/2$ gives the Catalan number $C_n$. Gessel and Xin \cite{GX} presented combinatorial interpretations for $S(n, 2)$ and $S(n,3)$, but in general it remains an intriguing open problem to find a combinatorial interpretation of the super Catalan numbers. In this paper, we give a combinatorial interpretation for $S(m, m+s)$ for $s \leq 4$.

\section{Another Super Catalan Identity}

To present the combinatorial interpretations, we first derive an identity for the super Catalan numbers from the well-known Von Szily identity (\cite[Section 6]{G})
\begin{equation}
\label{eq:vonszily}
S(m, n) = \sum_{k} (-1)^k \binom{2n}{n-k} \binom{2m}{m+k}.
\end{equation}
We provide both a combinatorial and an algebraic proof for this identity.

\begin{prop}
\label{prop:newid}
For $m, s \geq 0$, the following identity for the super Catalan numbers holds:
\begin{equation}
\label{eq:newid}
S(m, m + s) = \sum_{k} (-1)^k \binom{2m}{m-k} \binom{2s}{s + 2k}.
\end{equation}
\end{prop}

\noindent{\it Combinatorial Proof. }
We first interpret \eqref{eq:vonszily} in terms of lattice paths. We assume without loss of generality that $m \leq n$ and note that for any $k$, $\binom{2n}{n-k} \binom{2m}{m+k}$ counts the number of lattice paths from $(0,0)$  to $(m+n, m+n)$ going though the point $(m+k, m-k)$ with unit right and up steps.

We now define a sign-reversing involution $\phi$ on the lattice paths $$\bigcup_k  \left( (0,0) \rightarrow (m+k, m-k) \rightarrow (m+n, m+n) \right)$$ with the sign determined by the parity of $k$ such that the number of fixed paths under this involution is the super Catalan number. We let $P = (s_1, s_2, \ldots, s_{2m+2n})$, an ordered set, be a path from the point $(0,0)$ to $(m+n, m+n)$ where each $s_i = R$ or $U$ depending on whether it is a right step or an up step. We then find (if it exists) the least $i$, $1 \leq i \leq 2m$, such that $s_i \neq s_{2m+i}$ and switch these two steps. In other words, we have the map $\phi(P)= P^\prime = (s_1^\prime, s_2^\prime, \ldots, s_{2m+2n}^\prime)$ such that $$s_j^\prime = \begin{cases} s_{2m+i}, & j = i \\ s_{i}, & j = 2m + i \\ s_j, & \text{otherwise.} \end{cases}$$

For such a path, since $k$ is the number of right steps in the first $2m$ steps minus $m$, $P$ and $P^\prime$ will have $k$'s of opposite parity so $\phi$ is a sign-reversing involution. We notice that all the lattice paths with $s_k \ne s_{2m+k}$ for some $1 \le k \le 2m$ are cancelled under the involution $\phi.$ Therefore, the paths fixed under the involution are those from the point $(0,0)$ to $(m-k, m+k)$ to $(m+n, m+n)$ such that $s_i = s_{2m+i}$ for all $1\leq i \leq 2m$. For any $k$, we have $\binom{2m}{m-k} \binom{2n-2m}{n-m + 2k}$ such paths since by symmetry there are $\binom{2m}{m-k}$ ways to choose the first $4m$ steps and $\binom{2n-2m}{n-m+2k}$ ways to choose the last $2n-2m$ steps. Accounting for the parity of the $k$'s, we then have that $$S(m, n) = \sum_k (-1)^k \binom{2m}{m-k} \binom{2n-2m}{n-m+2k}.$$ Letting $s = n-m$ then gives us (4) as desired.
\qed \\
 
\noindent{\it Algebraic Proof. }
Equating the coefficients of $x^{m+n}$ in $$(1+x)^{2n} (1-x)^{2m} = (1-x^2)^{2m} (1+x)^{2n-2m},$$ we have that $$ \sum_{j} (-1)^{m + j} \binom{2n}{n-j} \binom{2m}{m + j} = \sum_{k} (-1)^{k} \binom{2m}{k} \binom{2n - 2m}{m + n - 2k}.$$ Applying Von Szily's identity \eqref{eq:vonszily} on the left hand side and replacing $k$ by $m-k$ on the right hand side, we get that $$(-1)^m S(n,m) = \sum_k (-1)^{m - k} \binom{2m}{m - k} \binom{2n-2m}{n-m+2k}.$$ The result follows by letting $s = n - m$.
\qed

\section{A Combinatorial Interpretation of $S(m, m+s)$ for $s \leq 3$}
\label{main}
\subsection{The Main Theorem}
In this subsection, we use the identity in Proposition \ref{prop:newid} to provide a combinatorial interpretation for $S(m, m+s)$ for $s \leq 3$. To state our result, we first need to define four line segments, see Figure~\ref{fig:4lines}:
 
$\ell_1$ connects $(0,1)$ and $(m - 1/2, m + 1/2)$,

 $\ell_2$ connects $(1,0)$ and $(m+ 1/2, m - 1/2)$,

 $\ell_3$ connects $(m-1, m+1)$ and $(m+s-1, m+s+1)$, and

$\ell_4$ connects $(m+1, m-1)$ and $(m+s+1, m+s-1)$.

\begin{figure}

\begin{center}
\includegraphics{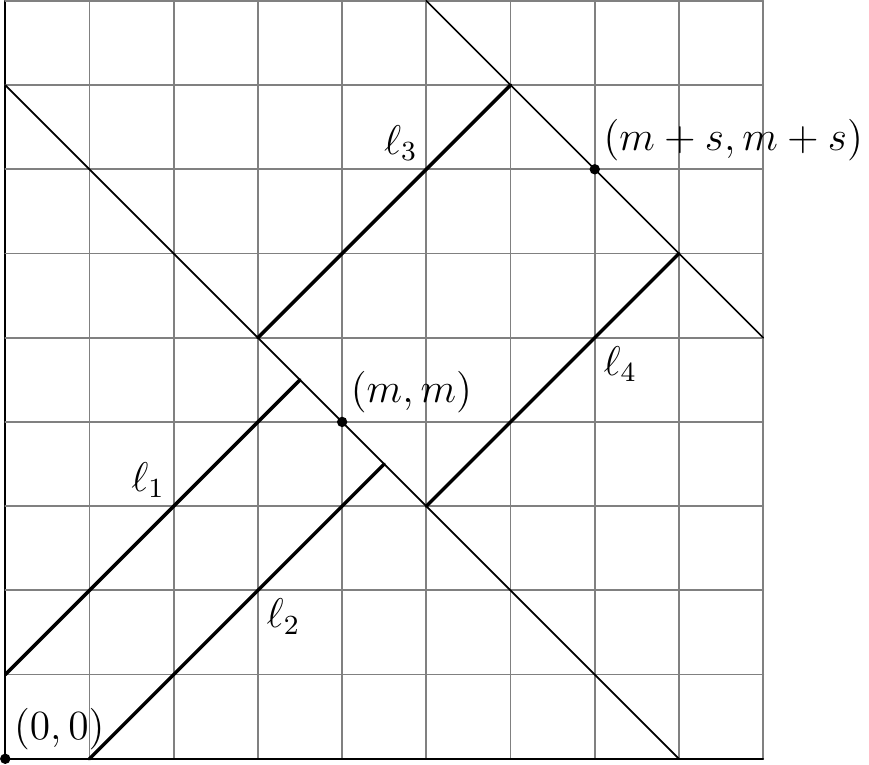}

\begin{center}
\caption{$\ell_1$ through $\ell_4$ for the $m=4, s = 3$ case.\label{fig:4lines}}
\end{center}

\end{center}
\end{figure}
\begin{thm}
\label{thm:main}
For $s \leq 3$, $S(m, m+s)$ counts the number of paths from $(0,0)$ to $(m+s, m+s)$ passing through $(m, m)$ that do not intersect both lines $\ell_1$ and $\ell_4$ or both $\ell_2$ and $\ell_3$.
\end{thm}

\begin{proof}
From \eqref{eq:newid}, we have that for $s \leq 3$, $$S(m, m+s) = \binom{2m}{m}\binom{2s}{s} - \binom{2m}{m-1} \binom{2s}{s + 2} - \binom{2m}{m+1} \binom{2s}{s-2}.$$
We notice that $\binom{2m}{m} \binom{2s}{s}$ counts the number of paths from the point $(0,0)$ to $(m, m)$ to $(m+s, m+s)$ and denote this set of paths by $\rm{Path}_0$. Similarly, $\binom{2m}{m-1}\binom{2s}{s+2}$ counts the number of paths from $(0,0)$ to $(m-1, m+1)$ to $(m+s+1, m+s-1)$ and $\binom{2m}{m+1}\binom{2s}{s-2}$ counts the number of paths from $(0,0)$ to $(m+1, m-1)$ to the point $(m+s-1, m+s+1)$. We denote these sets of paths by $\rm{Path}_{-1}$ and $\rm{Path}_{1}$, respectively.

We define an injection from $\rm{Path}_{-1} \cup \rm{Path}_{1}$ to $\rm{Path}_0$. To do so, first notice that any path $P \in \rm{Path}_{-1}$ must intersect $\ell_1$. We can find the last such intersection and reflect the tail of $P$'s $(0,0)$ to $(m-1, m+1)$  segment over $\ell_1$. This will give us a segment from $(0,0)$ to $(m, m)$. We then translate the last $2s$ steps of $P$ so that they start at $(m, m)$ and end at $(m+s+2, m+s-2)$. This segment must then intersect $\ell_4$. We can find the last such intersection and reflect the tail of this segment over $\ell_4$. Combining the two new segments, we now have a path in $\rm{Path}_0$, see Figure~\ref{fig:ex1} for an example. Notice that this map has a well-defined inverse. We define a similar map for paths in $\rm{Path}_{1}$, but reflect over lines $\ell_2$ and $\ell_3$ instead of $\ell_1$ and $\ell_4$.

We notice that paths in $\rm{Path}_{-1}$ cancel exactly with the paths in $\rm{Path}_0$ that intersect both lines $\ell_1$ and $\ell_4$. Similarly, we find that paths in $\rm{Path}_{1}$ cancel out exactly with the paths in $\rm{Path}_0$ that intersect both lines $\ell_2$ and $\ell_3$. Moreover, any path in $\rm{Path}_0$ does not intersect both $\ell_3$ and $\ell_4$ since $s \le 3$. This guarantees that the paths in $\rm{Path}_0$ that $\rm{Path}_{-1}$ and $\rm{Path}_{1}$ cancel do not overlap. Therefore, $S(m, m+s)$ counts the number of paths in $\rm{Path}_0$ that do not intersect both lines $\ell_1$ and $\ell_4$ or both $\ell_2$ and $\ell_3$.

\end{proof}

\begin{figure}
\begin{center}

\includegraphics{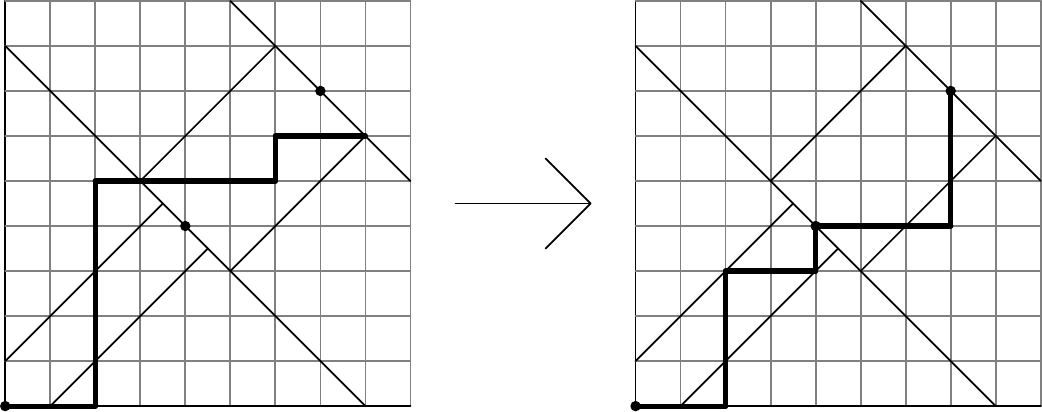}
\begin{center}

\caption{A mapping from Path$_{-1}$ to Path$_0$ for $m=4, s = 3$. \label{fig:ex1}}

\end{center}
\end{center}
\end{figure}

\subsection{Enumeration of the Paths in Theorem \ref{thm:main}}
From \eqref{eq:supercat}, we can find explicit expressions for the super Catalan numbers when $s$ is small. In this subsection, we count the paths remaining under the injection defined in Theorem \ref{thm:main} and show how they match with these explicit expressions.

For $s = 0$, we have that $$S(m, m) = \frac{(2m)! (2m!)}{m!m!(m+m)!} = \binom{2m}{m},$$ where the central binomial coefficient counts the paths from $(0,0)$ to $(m,m).$ These are exactly the paths specified in Theorem \ref{thm:main} for $s = 0$.
 
For $s=1$, we have that $$S(m, m+1) = \frac{(2m)!(2m + 2)!}{m!(m+1)! (2m+1)!} = 2 \binom{2m}{m},$$ which counts the number of paths from $(0,0)$ to $(m+1, m+1)$ going through $(m, m)$. Since no paths from $(m, m)$ to $(m+1, m+1)$ intersect $\ell_3$ or $\ell_4$, these are also exactly the paths specified in Theorem \ref{thm:main}.

For $s=2$,
 $$S(m, m + 2) = \frac{(2m)!(2m+4)!}{m!(m+2)! (2m + 2)!} = 2 (2m + 3) C_m.$$
By Theorem \ref{thm:main}, the paths remaining in $\rm{Path}_0$ under the injection either intersect only one of $\ell_1$ and $\ell_2$ or intersect both $\ell_1$ and $\ell_2$. We first count those that only intersect $\ell_1$. There are $C_m$ ways to choose the first $2m$ steps since this segment of the path must stay below the line $y=x$. Then, there are $5$ possible paths from $(m, m)$ to $(m + 2, m+2)$ that do not intersect $\ell_4$. Hence, we have totally $5C_m$ paths from $(0,0)$ to $(m,m)$ to $(m+2,m+2)$ that intersect only $\ell_1$. By symmetry, we have $5C_m$ paths that only intersect $\ell_2$.

We now count the paths that intersect both $\ell_1$ and $\ell_2$. Consider the segment from $(0,0)$ to $(m, m)$. We notice that to get the number of paths that intersect both $\ell_1$ and $\ell_2$, we subtract the number of paths that intersect only $\ell_1$ or $\ell_2$ from the total number of paths from $(0,0)$ to $(m, m)$. We have $$\binom{2m}{m} - 2C_m = \binom{2m}{m} - \frac{2}{m+1} \binom{2m}{m} = (m-1)C_m$$ such paths. Since there are $4$ paths from $(m, m)$ to $(m+2, m+2)$ that do not intersect $\ell_3$ or $\ell_4$, we have $4(m-1)C_m$ paths that intersect both $\ell_1$ and $\ell_2$. Therefore, the total number of paths is $$2 \cdot 5C_m + 4(m-1)C_m = 2(2m+3)C_m,$$ as desired.

Finally, for $s=3$, we have that $$S(m, m + 3) = \frac{(2m)!(2m+6)!}{m!(m+3)! (2m + 3)!} = 4 (2m + 5) C_m.$$ Applying a similar method as in the $s=2$ case, we can count $2 \cdot 14 C_m$ paths that only cross $\ell_1$ or $\ell_2$ and $8(m-1)C_m$ paths that cross both $\ell_1$ and $\ell_2$. It is clear that $$2 \cdot 14C_m + 8(m-1)C_m = 4(2m+5)C_m.$$

\section{A Combinatorial Interpretation of $S(m, m+4)$}
In general, the methods used for finding a combinatorial interpretation for the super Catalan numbers $S(m, m + s)$ for $s \leq 3$ in Section \ref{main} do not generalize nicely to higher $s$. In this section, we examine the case of $s = 4$.

We first define the following lines in addition to $\ell_1, \ell_2, \ell_3, \ell_4$ as defined in Section 3. These lines are pictured in Figure~\ref{fig:7lines}.

$\ell_5$ connects $(m - 1/2, m + 1/2)$ and $(m +s - 1/2, m +s + 1/2)$,

 $\ell_6$ connects $(m, m)$ and $(m+s, m+s)$, and

 $\ell_7$ connects $(m+1/2, m-1/2)$ and $(m+s+1/2, m+s-1/2)$.

\begin{figure}
\begin{center}
 \includegraphics{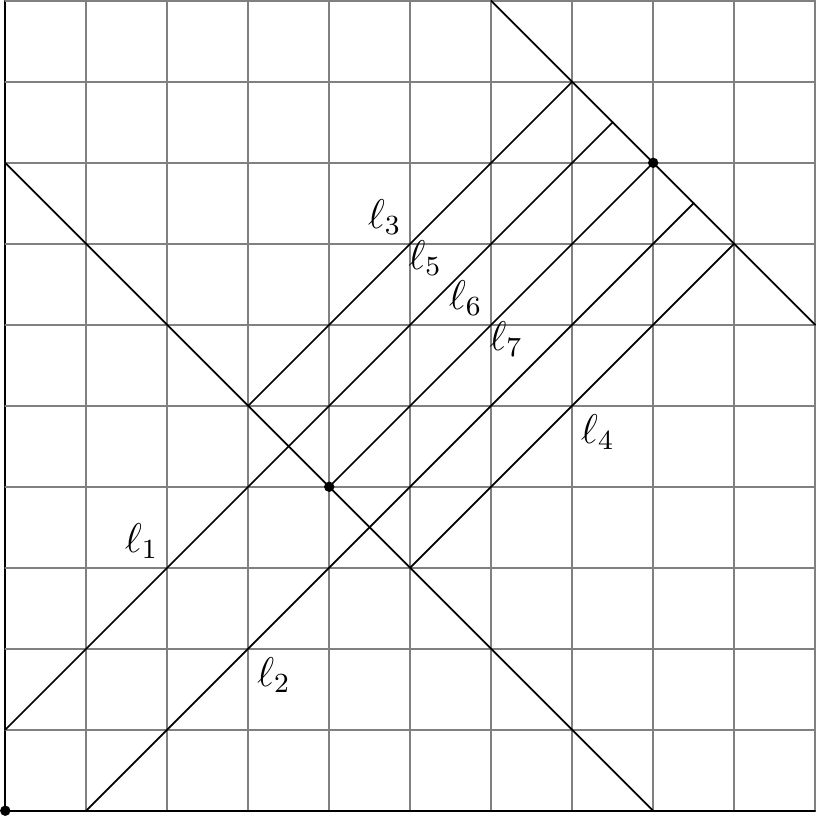}
\begin{center}
\caption{$\ell_1$ through $\ell_7$ for $m=4, s=4$. \label{fig:7lines}}
\end{center}
\end{center}

\end{figure} 

\begin{thm}
\label{thm:s4}
For $s = 4$, $S(m, m+s)$ counts the number of paths from $(0,0)$ to $(m,m)$ to $(m+s,m+s)$ that do not touch both lines $\ell_1$ and $\ell_4$ or $\ell_2$ and $\ell_3$ such that if they have an intersection of $\ell_1$ before an intersection of $\ell_2$, they do not also remain between lines $\ell_5$ and $\ell_6$ or between $\ell_6$ and $\ell_7$.
\end{thm}

\begin{proof}

For $s = 4$, we notice that (4) gives us five terms corresponding to $k, -2 \leq k \leq 2$. Interpreting these terms with lattice paths, we again have that for any $i, -2 \leq i \leq 2$, the $k = i$ term counts the number of paths in $\rm{Path}_i$ where $\rm{Path}_i$ is defined as the set of lattice paths from $(0,0)$ to $(m+i, m-i)$ to $(m+s-i, m+s+i)$.
We again map paths in $\rm{Path}_{1}$ and $\rm{Path}_{-1}$ to $\rm{Path}_{0}$ using the mapping defined in the proof of Theorem 3.1. Under this mapping however, we have double cancellation of paths that intersect all four lines $\ell_1, \ell_2, \ell_3,$ and $\ell_4$. Therefore,
\begin{align*}
|\rm{Path}_0| - |\rm{Path}_1| - |\rm{Path}_{-1}| =&  |\{ \text{Paths that do not intersect both $\ell_1, \ell_4$ or $\ell_2, \ell_3$}\}| \\
& - |\{\text{Paths that intersect $\ell_1, \ell_2, \ell_3,$ and $\ell_4$}\}|.
\end{align*}

We may also map the terms in $\rm{Path}_{2}$ and $\rm{Path}_{-2}$ into the set $\rm{Path}_0$. We first define $\ell_8$ to be the line segment from $(0,3)$ to $(m - 3/2, m + 3/2)$ and $\ell_9$ to be the line segment from $(3,0)$ to $(m+3/2, m-3/2)$. Then all paths in $\rm{Path}_{-2}$ must intersect $\ell_8$. We find the last such intersection and reflect the tail end of these first $2m$ steps over line $\ell_8$. Now we have a segment that intersects $\ell_8$ and ends at $(m-1, m+1)$. This segment must also intersect $\ell_1$ before its last intersection of $\ell_8$. We find the last such intersection of $\ell_1$ and reflect the tail end of the segment over $\ell_1$. This then gives us a segment from $(0,0)$ to $(m,m)$ that has an intersection of $\ell_2$ somewhere before an intersection of $\ell_1$.

We note that this is also an invertible operation. Also, we transform the last $2s$ steps of the path to the segment from $(m,m)$ to $(m + s, m+s)$ that intersects first $\ell_3$ and then $\ell_4$, see Figure~\ref{fig:ex2} for an example.

By symmetry, we can map the paths in $\rm{Path}_{2}$ to the paths in $\rm{Path}_0$ that have an intersection of $\ell_1$ before an intersection of $\ell_2$ and also hit lines $\ell_4$ and $\ell_3$ in that order. Hence, the paths in $\rm{Path}_{-2}$ and $\rm{Path}_2$ are mapped to exactly the paths in $\rm{Path}_0$ that hit all four lines $\ell_1, \ell_2, \ell_3, \ell_4$, but at some point hit $\ell_1$ before hitting $\ell_2$.
\begin{figure}
\begin{center}

\includegraphics{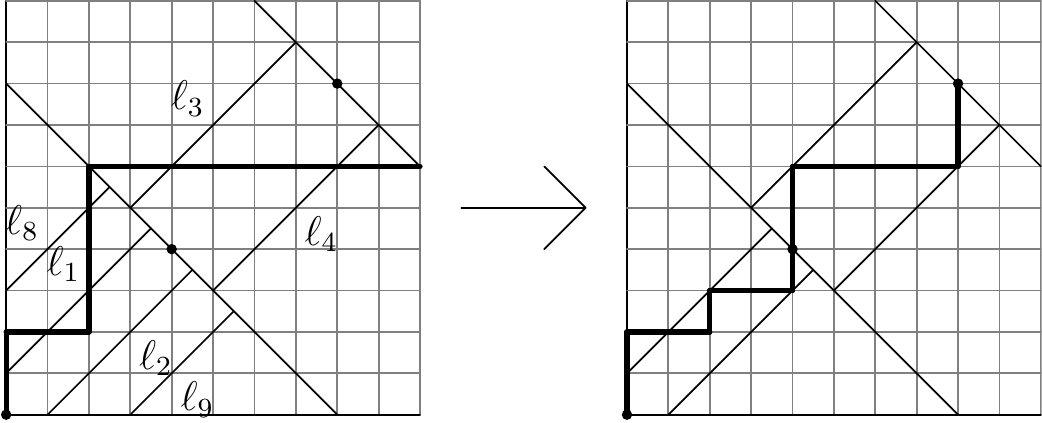}

\begin{center}
\caption{A transformation from Path$_{-2}$ to Path$_0$ for $m=4, s=4$. \label{fig:ex2}}
\end{center}
\end{center}
 \end{figure}
 
 Adding this to our count for $|\rm{Path}_0| - |\rm{Path}_1| - |\rm{Path}_{-1}|$, we have
\begin{align*}
S(m, m+4) = & |\{ \text{Paths that do not intersect both $\ell_1, \ell_4$ or $\ell_2, \ell_3$}\}| \\
& - |\{\text{Paths that intersect $\ell_1, \ell_2, \ell_3,$ and $\ell_4$, but never $\ell_1$ before $\ell_2$}\}|.
\end{align*}
 
 But if we define the sets $S_1$ and $S_2$ where $$S_1 = \{\text{Paths that intersect $\ell_1, \ell_2, \ell_3,$ and $\ell_4$, but not $\ell_1$ before $\ell_2$}\}$$ and
 \begin{align*}
 S_2 = \{ \text{Paths that intersect } & \text{$\ell_1$ and $\ell_2$, not $\ell_1$ before $\ell_2$} \\
 & \text{ that stay between $\ell_5$ and $\ell_6$ or $\ell_6$ and $\ell_7$}\},
 \end{align*}
 we can find a bijection between the two sets as follows. For paths in $S_1$, we map the last $2s$ steps of paths that intersect first $\ell_4$ and then $\ell_3$ to those that remain between $\ell_5$ and $\ell_6$, and the last $2s$ steps of paths that intersect first $\ell_3$ and then $\ell_4$ to those that remain between $\ell_6$ and $\ell_7$.
  
 Hence, we can find that $S(m, m+4)$ counts the number of paths from $(0,0)$ to $(m,m)$ to $(m+s,m+s)$ that do not touch both lines $\ell_1$ and $\ell_4$ or $\ell_2$ and $\ell_3$ such that if they have an intersection of $\ell_1$ before an intersection of $\ell_2$, they do not also remain between lines $\ell_5$ and $\ell_6$ or between $\ell_6$ and $\ell_7$.
 
 \end{proof}
 
 \section{Remarks}
The combinatorial interpretation of $S(m, m+4)$ given in Theorem 4.1 is not very satisfying because it imposes many conditions on the paths that are counted. The examination of the case of $s=4$, however, reveals some of the difficulties of using this method for $s \geq 4$, namely that we have to deal with double cancellation as well as the addition of more paths. For these reasons, it is unlikely that the methods used in the proof of Theorem 3.1 will generalize to higher $s$.

\section{Acknowledgments}
This research was conducted at the 2012 summer REU (Research Experience for Undergraduates) program at the University of Minnesota, Twin Cities, funded by NSF grants DMS-1148634 and DMS-1001933. The first author was also partially supported by the Carleton Kolenkow Reitz Fund. The authors would like to thank Profs. Dennis Stanton, Vic Reiner, Gregg Musiker, and Pavlo Pylyavskyy for mentoring the program. We would also like to express our particular gratitude to Prof. Stanton and Alex Miller for their guidance throughout the project.

\end{document}